\documentclass[11pt, final]{article}
\usepackage{a4}
\usepackage{amsmath}%
\usepackage{amstext}%
\usepackage{amssymb}%
\usepackage{showkeys}%
\usepackage{epsfig}%
\usepackage{cite}
\usepackage{tikz}
\usepackage{subfig}
\usepackage{caption}
\usepackage{multirow}
\usepackage{booktabs}
\usepackage{floatrow}

\usepackage{enumitem}
\usepackage{listings}
\newtheorem{theorem}{Theorem}

\newtheorem{algorithm}[theorem]{Algorithm}

\newtheorem{lemma}[theorem]{Lemma}

\newtheorem{pb}[theorem]{Problem}
\newenvironment{problem}{\pb\rm}{\endpb}

\newtheorem{rem}[theorem]{Remark}
\newenvironment{remark}{\rem\rm}{\endrem}

\newcounter{unnumber}

\newtheorem{prf}[unnumber]{Proof.}
\newenvironment{proof}{\prf\rm}{\hfill{$\blacksquare$}\endprf}
\newcommand{\R}{\mathbb{R}}%
\newcommand{\N}{\mathbb{N}}%
\newcommand{\ol}{\overline}%

\newcommand{\bx}{\ensuremath{\overline{x}}}

\newcommand{\bp}{\ensuremath{\overline{p}}}
\newcommand{\bv}{\ensuremath{\overline{v}}}
\newcommand{\bu}{\ensuremath{\overline{u}}}

\renewcommand{\>}{\right\rangle}
\newcommand{\<}{\left\langle}
\newcommand{\f}{\ensuremath{\boldsymbol}}
\newcommand{\fx}{\ensuremath{\boldsymbol{x}}}

\newcommand{\fbx}{\ensuremath{\boldsymbol{\overline{x}}}}

\newcommand{\fy}{\ensuremath{\boldsymbol{y}}}
\newcommand{\fz}{\ensuremath{\boldsymbol{z}}}

\newcommand{\fK}{\ensuremath{\boldsymbol{\mathcal{K}}}}
\newcommand{\h}{\ensuremath{\mathcal{H}}}
\newcommand{\g}{\ensuremath{\mathcal{G}}}
\renewcommand{\Box}{\ensuremath{\mbox{\small$\,\square\,$}}}

\DeclareMathOperator*\dom{dom}%
\DeclareMathOperator*\B{\overline{\R}}%
\DeclareMathOperator*\gr{Gr}%
\DeclareMathOperator*\ran{ran}%
\DeclareMathOperator*\id{Id}%
\DeclareMathOperator*\prox{prox}%

\DeclareMathOperator*\zer{zer}
\DeclareMathOperator*\fix{Fix}

\title{Inertial Douglas-Rachford splitting for monotone inclusion problems}

\author{Radu Ioan Bo\c{t} \thanks{University of Vienna, Faculty of Mathematics, Oskar-Morgenstern-Platz 1, A-1090 Vienna, Austria, 
email: radu.bot@univie.ac.at. Research partially supported by DFG (German Research Foundation), project BO 2516/4-1.} \and 
Ern\"{o} Robert Csetnek \thanks {University of Vienna, Faculty of Mathematics, Oskar-Morgenstern-Platz 1, A-1090 Vienna, Austria, 
email: ernoe.robert.csetnek@univie.ac.at. Research supported by DFG (German Research Foundation), project BO 2516/4-1.} 
 \and Christopher Hendrich \thanks{Department of Mathematics, Chemnitz University of Technology, D-09107 Chemnitz, Germany, e-mail: 
 christopher.hendrich@mathematik.tu-chemnitz.de. Research supported by a Graduate Fellowship of the Free State Saxony, Germany.}}

\begin{document}
\maketitle

\noindent \textbf{Abstract.} We propose an inertial Douglas-Rachford splitting algorithm for finding the set of zeros of the sum 
of two maximally monotone operators in Hilbert spaces and investigate its convergence properties. To this end we formulate first the inertial version of the Krasnosel'ski\u{\i}--Mann algorithm for approximating the set 
of fixed points of a nonexpansive operator, for which we also provide an exhaustive convergence analysis. By using  a product space approach we employ these results to the solving of monotone inclusion problems 
involving linearly composed and parallel-sum type  operators and provide in this way iterative schemes where each of the maximally monotone mappings is accessed separately via its resolvent. We consider also the 
special instance of solving a primal-dual pair of nonsmooth convex optimization problems and illustrate the theoretical results via some numerical experiments in clustering and location theory. 
\vspace{1ex}

 \vspace{1ex}


\section{Introduction and preliminaries}\label{sec-intr}

The problem of approaching the set of zeros of maximally monotone operators by means of splitting iterative algorithms, where  each of the operators involved is evaluated separately, either via its resolvent in the set-valued 
case, or by means of the operator itself in the single-valued case, continues to be a very attractive research area. This is due to its applicability in the context of solving real-life problems which can 
be modeled as nondifferentiable convex optimization problems, like those arising in image processing, signal recovery, support vector machines classification, location theory, clustering, network communications, etc.

In this manuscript we focus our attention on the \textit{Douglas-Rachford algorithm} which approximates the set of zeros of the sum of 
two maximally monotone operators. Introduced in \cite{dor} in the particular context of solving matrix equations, its convergence properties have been investigated also in \cite{lions-mercier}. One of the methods 
for proving the convergence of the classical Douglas--Rachford splitting method is by treating it as a particular case of the  Krasnosel'ski\u{\i}--Mann algorithm designed for finding fixed points of nonexpansive operators 
(see \cite{bauschke-book}, \cite{combettes}). This approach has the advantage to allow the inclusion of relaxation parameters in the update rules of the iterates. 

In this paper we introduce and investigate the convergence properties of an inertial Douglas-Rachford splitting algorithm. Inertial proximal methods go back to \cite{alvarez2000, alvarez-attouch2001}, 
where it has been noticed that the discretization of a differential system of second-order in time gives rise to a generalization of the classical proximal-point algorithm for finding the zeros of a maximally monotone operator 
(see \cite{rock-prox}), nowadays called \textit{inertial proximal-point algorithm}. One  of the main features of the inertial proximal algorithm is that the next iterate is defined by making use of the last two iterates. 
Since its introduction, one can notice an increasing interest in algorithms having this particularity, see \cite{alvarez2004, alvarez-attouch2001,
cabot-frankel2011, mainge2008, mainge-moudafi2008, moudafi-oliny2003, b-c-inertial, att-peyp-red}. 

In order to prove the convergence of the proposed inertial Douglas-Rachford algorithm we formulate first the inertial version of the 
Krasnosel'ski\u{\i}--Mann algorithm for approximating the set of fixed points of a nonexpansive operator and investigate its convergence properties, the obtained results having their own interest. The convergence of the inertial Douglas-Rachford scheme is then derived by applying the inertial version 
of the Krasnosel'ski\u{\i}--Mann algorithm to the composition of the reflected resolvents of the maximally monotone 
operators involved in the monotone inclusion problem. 

The second major aim of the paper is to make use of these results when formulating an \textit{inertial Douglas-Rachford 
primal-dual algorithm} designed to solve monotone inclusion problems involving linearly composed and parallel-sum type operators. Let us mention that the classical Douglas-Rachford algorithm cannot handle monotone inclusion problems where some of the 
set-valued mappings involved are composed with linear continuous operators, since in general there is no closed form for 
the resolvent of the composition. The same applies in the case of monotone inclusion problems involving parallel-sum type  operators. \textit{Primal-dual methods} are modern techniques which overcome this drawback, 
having as further highlights their full decomposability and the fact that they are able to solve concomitantly a primal monotone inclusion problem and 
its dual one in the sense of Attouch-Th\'{e}ra \cite{at-th} (see \cite{b-c-h1, b-c-h2, b-h, b-h2, br-combettes, combettes-pesquet, vu, ch-pck, condat2013}  for more details). 

The structure of the paper is the following. In the remainder of this section we recall some elements of the theory of maximal monotone operators and some convergence results needed in the paper.
The next section contains the inertial type of the Krasnosel'ski\u{\i}--Mann scheme followed by the inertial 
Douglas-Rachford algorithm with corresponding weak and strong convergence results. In Section \ref{sec3} we formulate the  inertial 
Douglas-Rachford primal-dual splitting algorithm and study its convergence, while in Section \ref{sec4} we make use of this iterative scheme when solving primal-dual pairs of convex optimization problems. We close the paper by illustrating the theoretical results via some numerical experiments in clustering and location theory. 

For the notions and results presented as follows we refer the reader to \cite{bo-van, b-hab, bauschke-book, EkTem, simons, Zal-carte}. Let $\N= \{0,1,2,...\}$ be the set of nonnegative integers.
Let ${\cal H}$ be a real Hilbert space with \textit{inner product} $\langle\cdot,\cdot\rangle$ and associated \textit{norm} $\|\cdot\|=\sqrt{\langle \cdot,\cdot\rangle}$.
The symbols $\rightharpoonup$ and $\rightarrow$ denote weak and strong convergence, respectively. The following identity will be used  
several times in the paper (see \cite[Corollary 2.14]{bauschke-book}):
\begin{equation}\label{id-hilb} \|\alpha x+(1-\alpha)y\|^2+\alpha(1-\alpha)\|x-y\|^2=\alpha\|x\|^2+(1-\alpha)\|y\|^2 \ \forall \alpha\in\R
 \ \forall (x,y)\in{\cal H}\times{\cal H}.\end{equation}

When ${\cal G}$ is another Hilbert space and $K:{\cal H} \rightarrow {\cal G}$ a linear continuous operator,
then the \textit{norm} of $K$ is defined as $\|K\| = \sup\{\|Kx\|: x \in {\cal H}, \|x\| \leq 1\}$,
while $K^* : {\cal G} \rightarrow {\cal H}$, defined by $\langle K^*y,x\rangle = \langle y,Kx \rangle$ for all $(x,y) \in {\cal H} \times {\cal G}$, denotes the \textit{adjoint operator} of $K$.

For an arbitrary set-valued operator $A:{\cal H}\rightrightarrows {\cal H}$ we denote by $\gr A=\{(x,u)\in {\cal H}\times {\cal H}:u\in Ax\}$ its \emph{graph}, by $\dom A=\{x \in {\cal H} : Ax \neq \varnothing\}$
its \emph{domain}, by $\ran A=\cup_{x\in{\cal{H}}} Ax$ its {\it range} and by $A^{-1}:{\cal H}\rightrightarrows {\cal H}$ its
\emph{inverse operator}, defined by $(u,x)\in\gr A^{-1}$ if and only if $(x,u)\in\gr A$.
We use also the notation $\zer A=\{x\in{\cal{H}}:0\in Ax\}$ for the \emph{set of zeros} of $A$. We say that $A$ is \emph{monotone}
if $\langle x-y,u-v\rangle\geq 0$ for all $(x,u),(y,v)\in\gr A$. A monotone operator $A$ is said to be \emph{maximally monotone}, if there exists no proper monotone extension of the graph of $A$ on ${\cal H}\times {\cal H}$.
The \emph{resolvent} of $A$ is
$$J_A:{\cal H} \rightrightarrows {\cal H}, J_A=(\id\nolimits_{{\cal H}}+A)^{-1},$$ 
and the \textit{reflected resolvent} of $A$ is
$$R_A:{\cal H} \rightrightarrows {\cal H}, R_A=2J_A-\id\nolimits_{{\cal H}},$$
where $\id_{{\cal H}} :{\cal H} \rightarrow {\cal H}, \id_{\cal H}(x) = x$ for all $x \in {\cal H}$, is the \textit{identity operator} on ${\cal H}$. 
Moreover, if $A$ is maximally monotone, then $J_A:{\cal H} \rightarrow {\cal H}$ is single-valued and maximally monotone
(see \cite[Proposition 23.7 and Corollary 23.10]{bauschke-book}). For an arbitrary $\gamma>0$ we have (see \cite[Proposition 23.2]{bauschke-book})
$$p\in J_{\gamma A}x \ \mbox{if and only if} \ (p,\gamma^{-1}(x-p))\in\gr A$$
and (see \cite[Proposition 23.18]{bauschke-book})
\begin{equation}\label{j-inv-op}
J_{\gamma A}+\gamma J_{\gamma^{-1}A^{-1}}\circ \gamma^{-1}\id\nolimits_{{\cal H}}=\id\nolimits_{{\cal H}}.
\end{equation}

Further, let us mention some classes of operators that are used in the paper. The operator $A$ is said to be \textit{uniformly monotone} if there exists an increasing function
$\phi_A : [0,+\infty) \rightarrow [0,+\infty]$ that vanishes only at $0$, and
\begin{equation}\label{unif-mon-def}\langle x-y,u-v \rangle \geq \phi_A \left( \| x-y \|\right) \ \forall (x,u),(y,v) \in \gr A.\end{equation} 
Prominent representatives 
of the class of uniformly monotone operators are the strongly monotone operators.
Let $\gamma>0$ be arbitrary. We say that  $A$ is \textit{$\gamma$-strongly monotone}, if $\langle x-y,u-v\rangle\geq \gamma\|x-y\|^2$ for all 
$(x,u),(y,v)\in\gr A$.

We consider also the class of nonexpansive operators. An operator $T:D\rightarrow {\cal H}$, where $D\subseteq {\cal H}$ is nonempty, is said to be 
\textit{nonexpansive}, if $\|Tx-Ty\|\leq\|x-y\|$ for all $x,y\in D$. We use the notation $\fix T=\{x\in D:Tx=x\}$ for the set of 
\textit{fixed points} of $T$. Let us mention that the resolvent and the reflected resolvent of a maximally monotone operator are both nonexpansive (see 
\cite[Corollary 23.10]{bauschke-book}). 

The following result, which is a consequence of the demiclosedness principle (see \cite[Theorem 4.17]{bauschke-book}), 
will be useful in the proof of the convergence of the inertial version of the Krasnosel'ski\u{\i}--Mann algorithm. 

\begin{lemma}\label{x_n-T-fix}(see \cite[Corollary 4.18]{bauschke-book}) Let $D\subseteq {\cal H}$ be nonempty closed and convex, $T:D\rightarrow {\cal H}$ be nonexpansive and
let $(x_n)_{n\in\N}$ be a sequence in $D$ and $x\in {\cal H}$ such that $x_n\rightharpoonup x$ and $Tx_n-x_n\rightarrow 0$ as $n\rightarrow +\infty$. 
Then $x\in\fix T$.
\end{lemma}

The \textit{parallel sum} of two operators $A,B:{\cal H}\rightrightarrows {\cal H}$ is defined by 
$A\Box B: {\cal H}\rightrightarrows {\cal H}, A\Box B=(A^{-1}+B^{-1})^{-1}$. If $A$ and $B$ are monotone, then we have the following characterization of the set of zeros 
of their sum (see \cite[Proposition 25.1(ii)]{bauschke-book}): 
\begin{equation}\label{zer(A+B)} \zer(A+B)=J_{\gamma B}(\fix R_{\gamma A}R_{\gamma B}) \ \forall \gamma >0.  
\end{equation}

The following result is a direct consequence of \cite[Corollary 25.5]{bauschke-book} and will be used in the proof of the convergence of 
the inertial Douglas--Rachford splitting algorithm. 

\begin{lemma}\label{conv-weak-str-sum} Let $A,B:{\cal H}\rightrightarrows{\cal H}$ be maximally monotone operators and the sequences 
$(x_n,u_n)_{n\in\N}\in\gr A$, $(y_n,v_n)_{n\in\N}\in\gr B$ such that 
$x_n\rightharpoonup x, u_n\rightharpoonup u, y_n\rightharpoonup y, v_n\rightharpoonup v$, $u_n+v_n\rightarrow 0$ and $x_n-y_n\rightarrow 0$ as 
$n\rightarrow+\infty$. Then $x=y\in\zer(A+B)$, $(x,u)\in\gr A$ and $(y,v)\in\gr B$.
\end{lemma}

We close this section by presenting two convergence results which will be crucial for the proof of the main results
in the next section. 

\begin{lemma}\label{ext-fejer1} (see \cite{alvarez-attouch2001, alvarez2000, alvarez2004}) Let $(\varphi_n)_{n\in\N}, (\delta_n)_{n\in\N}$ and $(\alpha_n)_{n\in \N}$ be sequences in
$[0,+\infty)$ such that $\varphi_{n+1}\leq\varphi_n+\alpha_n(\varphi_n-\varphi_{n-1})+\delta_n$
for all $n \geq 1$, $\sum_{n\in \N}\delta_n< + \infty$ and there exists a real number $\alpha$ with
$0\leq\alpha_n\leq\alpha<1$ for all $n\in\N$. Then the following hold: \begin{itemize}\item[(i)] $\sum_{n \geq 1}[\varphi_n-\varphi_{n-1}]_+< + \infty$, where
$[t]_+=\max\{t,0\}$; \item[(ii)] there exists $\varphi^*\in[0,+\infty)$ such that $\lim_{n\rightarrow+\infty}\varphi_n=\varphi^*$.\end{itemize}
\end{lemma}

Finally, we recall a well known result on weak convergence in Hilbert spaces.

\begin{lemma}\label{opial} (Opial) Let $C$ be a nonempty set of ${\cal H}$ and $(x_n)_{n\in\N}$ be a sequence in ${\cal H}$ such that
the following two conditions hold: \begin{itemize}\item[(a)] for every $x\in C$, $\lim_{n\rightarrow + \infty}\|x_n-x\|$ exists;
\item[(b)] every sequential weak cluster point of $(x_n)_{n\in\N}$ is in $C$;\end{itemize}
Then $(x_n)_{n\in\N}$ converges weakly to a point in $C$.
 \end{lemma}

\section{An inertial Douglas--Rachford splitting algorithm}\label{sec2}

This section is dedicated to the formulation of an inertial Douglas--Rachford splitting algorithm which approaches the set of zeros
of the sum of two maximally monotone operators and to the investigation of its convergence properties. 

In the first part we propose an inertial version of the Krasnosel'ski\u{\i}--Mann algorithm for approximating the 
set of fixed points of a nonexpansive operator, a result which has its own interest. Notice that due to the presence of affine combinations in the iterative scheme, we 
have to restrict the setting to nonexpansive operators defined on affine subspaces. Let us underline that this assumption is fulfilled 
when considering the composition of the reflected resolvents of maximally monotone operators, which will be the case when deriving the inertial Douglas--Rachford algorithm.   

\begin{theorem}\label{inertial-kr-m-th} Let $M$ be a nonempty closed and affine subset of ${\cal H}$ and $T:M\rightarrow M$ a nonexpansive 
operator such that $\fix T\neq\varnothing$. We consider the following iterative scheme: 
\begin{equation}\label{inertial-kr-m-alg}x_{n+1}=x_n+\alpha_n(x_n-x_{n-1})+\lambda_n\Big[T\big(x_n+\alpha_n(x_n-x_{n-1})\big)-x_n-\alpha_n(x_n-x_{n-1})\Big] \ \forall n\geq 1\end{equation}
where $x_0,x_1$ are arbitrarily chosen in $M$, $(\alpha_n)_{n\geq 1}$ is nondecreasing with  $\alpha_1=0$ and $0\leq \alpha_n\leq\alpha < 1$ for every $n\geq 1$ and $\lambda, \sigma, \delta >0$ are such that 
\begin{equation}\label{assum}
\delta>\frac{\alpha^2(1+\alpha)+\alpha\sigma}{1-\alpha^2} \ \mbox{and} \ 0<\lambda\leq\lambda_n\leq \frac{\delta-\alpha\Big[\alpha(1+\alpha)+\alpha\delta+\sigma\Big]}{\delta\Big[1+\alpha(1+\alpha)+\alpha\delta+\sigma\Big]} \
\forall n \geq 1.
\end{equation}
Then the following statements are true:
\begin{itemize}\item[(i)] $\sum_{n\in\N}\|x_{n+1}-x_n\|^2<+\infty$; 
\item[(ii)] $(x_n)_{n\in\N}$ converges weakly to a point in $\fix T$.\end{itemize}
\end{theorem}

\begin{proof} Let us start with the remark that, due to the choice of $\delta$, $\lambda_n\in(0,1)$ for every $n\geq 1$. Further, we would like to notice that, since $M$ is affine, 
the iterative scheme provides a well-defined sequence in $M$.

(i) We denote
$$w_n:=x_n+\alpha_n(x_n-x_{n-1}) \ \forall n \geq 1.$$ Then the iterative scheme reads for every $n \geq 1$:
\begin{equation}\label{w_n} x_{n+1}=w_n+\lambda_n(Tw_n-w_n).\end{equation}
Let us fix an element $y\in\fix T$ and $n\geq 1$. It follows from \eqref{id-hilb} and the nonexpansiveness of $T$ that 
\begin{align}\label{id-T}
\|x_{n+1}-y\|^2 & =(1-\lambda_n)\|w_n-y\|^2+\lambda_n\|Tw_n-Ty\|^2-\lambda_n(1-\lambda_n)\|Tw_n-w_n\|^2\nonumber\\
                & \leq \|w_n-y\|^2-\lambda_n(1-\lambda_n)\|Tw_n-w_n\|^2.
\end{align}

Applying again \eqref{id-hilb} we have
\begin{align*}\|w_n-y\|^2 & =\|(1+\alpha_n)(x_n-y)-\alpha_n(x_{n-1}-y)\|^2\\ 
                          & = (1+\alpha_n)\|x_n-y\|^2-\alpha_n\|x_{n-1}-y\|^2+\alpha_n(1+\alpha_n)\|x_n-x_{n-1}\|^2,\end{align*}
hence by \eqref{id-T} we obtain 
\begin{align}\label{ineq} \|x_{n+1}-y\|^2-(1+\alpha_n)\|x_n-y\|^2+\alpha_n\|x_{n-1}-y\|^2
\leq & -\lambda_n(1-\lambda_n)\|Tw_n-w_n\|^2 \nonumber \\
     & +\alpha_n(1+\alpha_n)\|x_n-x_{n-1}\|^2.
 \end{align}
                          
Further, we have 
\begin{align}\label{ineq2} \|Tw_n-w_n\|^2
& = \left\|\frac{1}{\lambda_n}(x_{n+1}-x_n)+\frac{\alpha_n}{\lambda_n}(x_{n-1}-x_n)\right\|^2\nonumber\\
& = \frac{1}{\lambda_n^2}\|x_{n+1}-x_n\|^2+\frac{\alpha_n^2}{\lambda_n^2}\|x_n-x_{n-1}\|^2 
+ 2\frac{\alpha_n}{\lambda_n^2}\langle x_{n+1}-x_n,x_{n-1}-x_n\rangle\nonumber\\
& \geq \frac{1}{\lambda_n^2}\|x_{n+1}-x_n\|^2+\frac{\alpha_n^2}{\lambda_n^2}\|x_n-x_{n-1}\|^2\nonumber\\
& + \frac{\alpha_n}{\lambda_n^2}\left(-\rho_n\|x_{n+1}-x_n\|^2-\frac{1}{\rho_n}\|x_n-x_{n-1}\|^2\right), 
\end{align}
where we denote $\rho_n:=\frac{1}{\alpha_n+\delta\lambda_n}$.

We derive from \eqref{ineq} and \eqref{ineq2} the inequality (notice that $\lambda_n\in(0,1)$) 
\begin{align}\label{ineq3} \|x_{n+1}-y\|^2-(1+\alpha_n)\|x_n-y\|^2+\alpha_n\|x_{n-1}-y\|^2 \leq & \frac{(1-\lambda_n)(\alpha_n\rho_n-1)}{\lambda_n}\|x_{n+1}-x_n\|^2 \nonumber\\
& +\gamma_n\|x_n-x_{n-1}\|^2,
 \end{align}
where \begin{equation}\label{gamma_n}\gamma_n:=\alpha_n(1+\alpha_n)+\alpha_n(1-\lambda_n)\frac{1-\rho_n\alpha_n}{\rho_n\lambda_n}>0.\end{equation}

Taking again into account the choice of $\rho_n$ we have $$\delta = \frac{1-\rho_n\alpha_n}{\rho_n\lambda_n}$$ and by 
\eqref{gamma_n} it follows
\begin{equation}\label{ineq6}\gamma_n= \alpha_n(1+\alpha_n)+\alpha_n(1-\lambda_n)\delta\leq\alpha(1+\alpha)+\alpha\delta 
 \ \forall n\geq 1.\end{equation}
 
In the following we use some techniques from \cite{alvarez-attouch2001} adapted to our setting. We define the sequences
$\varphi_n:=\|x_n-y\|^2$ for all $n\in\N$ and $\mu_n:=\varphi_n-\alpha_n\varphi_{n-1}+\gamma_n\|x_n-x_{n-1}\|^2$ for all
$n\geq 1$. Using the monotonicity of $(\alpha_n)_{n \geq 1}$ and the fact that $\varphi_n\geq0$ for all $n\in\N$, we get
\begin{align*}
\mu_{n+1}-\mu_n\leq \varphi_{n+1}-(1+\alpha_n)\varphi_n+\alpha_n\varphi_{n-1}+\gamma_{n+1}\|x_{n+1}-x_n\|^2-\gamma_n\|x_n-x_{n-1}\|^2,
\end{align*}
which gives by \eqref{ineq3}
\begin{equation}\label{ineq4}
\mu_{n+1}-\mu_n\leq\left(\frac{(1-\lambda_n)(\alpha_n\rho_n-1)}{\lambda_n}+\gamma_{n+1}\right)\|x_{n+1}-x_n\|^2 \ \forall n \geq 1.
\end{equation}

We claim that
\begin{equation}\label{ineq5}\frac{(1-\lambda_n)(\alpha_n\rho_n-1)}{\lambda_n}+\gamma_{n+1}\leq-\sigma \ \forall n \geq 1.\end{equation}
Let be $n \geq 1$. Indeed, by the choice of $\rho_n$, we get 
\begin{align*}
& \ \frac{(1-\lambda_n)(\alpha_n\rho_n-1)}{\lambda_n}+\gamma_{n+1}\leq-\sigma\\
\Longleftrightarrow & \ \lambda_n(\gamma_{n+1}+\sigma)+(\alpha_n\rho_n-1)(1-\lambda_n)\leq 0\\
\Longleftrightarrow & \ \lambda_n(\gamma_{n+1}+\sigma)-\frac{\delta\lambda_n(1-\lambda_n)}{\alpha_n+\delta\lambda_n}\leq 0\\
\Longleftrightarrow & \ (\alpha_n+\delta\lambda_n)(\gamma_{n+1}+\sigma)+\delta\lambda_n\leq \delta.
\end{align*}

Thus, by using \eqref{ineq6}, we have 
$$(\alpha_n+\delta\lambda_n)(\gamma_{n+1}+\sigma)+\delta\lambda_n\leq
(\alpha+\delta\lambda_n)\Big(\alpha(1+\alpha)+\alpha\delta+\sigma\Big)+\delta\lambda_n\leq\delta,$$
where the last inequality follows by taking into account the upper bound considered for $(\lambda_n)_{n\geq 1}$ in \eqref{assum}. Hence the claim in  \eqref{ineq5} is true. 

We obtain from \eqref{ineq4} and \eqref{ineq5} that
\begin{equation}\label{ineq7}\mu_{n+1}-\mu_n\leq-\sigma\|x_{n+1}-x_n\|^2 \ \forall n\geq 1.\end{equation}

The sequence $(\mu_n)_{n \geq 1}$ is nonincreasing and the bound for $(\alpha_n)_{n \geq 1}$ delivers
\begin{equation}\label{ineq8}
-\alpha\varphi_{n-1}\leq\varphi_n-\alpha\varphi_{n-1}\leq \mu_n\leq\mu_1 \ \forall n \geq 1.
\end{equation}

We obtain
$$\varphi_n\leq \alpha^n\varphi_0+\mu_1\sum_{k=0}^{n-1}\alpha^k\leq \alpha^n\varphi_0+\frac{\mu_1}{1-\alpha} \ \forall n \geq 1,$$
where we notice that $\mu_1=\varphi_1\geq 0$ (due to the relation $\alpha_1=0$). Combining \eqref{ineq7} and \eqref{ineq8}, we get for all $n \geq 1$
$$\sigma\sum_{k=1}^{n}\|x_{k+1}-x_k\|^2\leq \mu_1-\mu_{n+1}\leq \mu_1+\alpha\varphi_n\leq \alpha^{n+1}\varphi_0+\frac{\mu_1}{1-\alpha},$$ which
shows that $\sum_{n \in \N}\|x_{n+1}-x_n\|^2<+\infty$.

(ii) We prove this by using the result of Opial given in Lemma \ref{opial}. We have proven above that for an arbitrary $y\in\fix T$ the inequality 
\eqref{ineq3} is true. By part (i), \eqref{ineq6} and Lemma \ref{ext-fejer1} we derive that $\lim_{n\rightarrow+\infty}\|x_n-y\|$ 
exists (we take into consideration also that in \eqref{ineq3} $\alpha_n\rho_n<1$ for all $n\geq 1$). On the other hand, 
let $x$ be a sequential weak cluster point of $(x_n)_{n\in\N}$, that is, the latter has a subsequence $(x_{n_k})_{k \in \N}$
fulfilling $x_{n_k}\rightharpoonup x$ as $k\rightarrow+\infty$. By part (i), the definition of $w_n$ and the upper bound 
requested for $(\alpha_n)_{n\geq 1}$, we get $w_{n_k}\rightharpoonup x$ as $k\rightarrow+\infty$. Further, by \eqref{w_n} we have 
\begin{equation}\label{Tw_n}\|Tw_n-w_n\|=\frac{1}{\lambda_n}\|x_{n+1}-w_n\|\leq \frac{1}{\lambda}\|x_{n+1}-w_n\|\leq \frac{1}{\lambda}\big(\|x_{n+1}-x_n\|+\alpha\|x_n-x_{n-1}\|\big),\end{equation}
thus by (i) we obtain $Tw_{n_k}-w_{n_k}\rightarrow 0$ as $k\rightarrow+\infty$. Applying now Lemma \ref{x_n-T-fix} for the 
sequence $(w_{n_k})_{k\in\N}$ we conclude that  $x\in\fix T$. Since the two assertions of Lemma \ref{opial} are verified, it follows that $(x_n)_{n\in\N}$ converges weakly to 
a point in $\fix T$. 
\end{proof}

\begin{remark} The condition $\alpha_1=0$ was imposed in order to ensure $\mu_1\geq 0$, which is needed in the proof.
An alternative is to require that $x_0=x_1$, in which case the assumption $\alpha_1=0$ is not anymore necessary.
\end{remark}

\begin{remark} Assuming that $\alpha=0$ (which enforces $\alpha_n=0$ for all $ n\geq 1$), the iterative scheme in the previous theorem is nothing else than the one in
the classical Krasnosel'ski\u{\i}--Mann algorithm:
\begin{equation}\label{kr-m-alg}x_{n+1}=x_n+\lambda_n(Tx_n-x_n) \ \forall n\geq 1.\end{equation}
Let us mention that the convergence of this iterative scheme can be proved under more general hypotheses, namely when 
$M$ is a nonempty closed and convex set and the sequence $(\lambda_n)_{n\in\N}$ satisfies the relation 
$\sum_{n\in\N}\lambda_n(1-\lambda_n)=+\infty$ (see \cite[Theorem 5.14]{bauschke-book}).
\end{remark}

We are now in position to state the inertial Douglas--Rachford splitting algorithm and to present its convergence properties. 

\begin{theorem}\label{inertial-DR} (Inertial Douglas--Rachford splitting algorithm) Let $A,B:{\cal H}\rightrightarrows{\cal H}$ be 
maximally monotone operators such that $\zer(A+B)\neq\varnothing$. Consider the following iterative scheme: 
$$(\forall n\geq 1)\hspace{0.2cm}\left\{
\begin{array}{ll}
y_n=J_{\gamma B}[x_n+\alpha_n(x_n-x_{n-1})]\\
z_n=J_{\gamma A}[2y_n-x_n-\alpha_n(x_n-x_{n-1})]\\
x_{n+1}=x_n+\alpha_n(x_n-x_{n-1})+\lambda_n(z_n-y_n)
\end{array}\right.$$
where $\gamma>0$, $x_0,x_1$ are arbitrarily chosen in $\cal{H}$, $(\alpha_n)_{n\geq 1}$ is nondecreasing with  $\alpha_1=0$ and $0\leq \alpha_n\leq\alpha < 1$ for every $n\geq 1$ and $\lambda, \sigma, \delta >0$ are such that 
\begin{equation*}
\delta>\frac{\alpha^2(1+\alpha)+\alpha\sigma}{1-\alpha^2} \ \mbox{and} \ 0<\lambda\leq\lambda_n\leq 2\frac{\delta-\alpha\Big[\alpha(1+\alpha)+\alpha\delta+\sigma\Big]}{\delta\Big[1+\alpha(1+\alpha)+\alpha\delta+\sigma\Big]} \
\forall n \geq 1.
\end{equation*}
Then there exists $x\in\fix(R_{\gamma A}R_{\gamma B})$ such that 
the following statements are true:
\begin{itemize} \item[(i)] $J_{\gamma B}x\in\zer(A+B)$;
\item[(ii)] $\sum_{n\in\N}\|x_{n+1}-x_n\|^2<+\infty$; 
\item[(iii)] $(x_n)_{n\in\N}$ converges weakly to $x$;
\item[(iv)] $y_n-z_n\rightarrow 0$ as $n\rightarrow+\infty$;
\item[(v)] $(y_n)_{n\geq 1}$ converges weakly to $J_{\gamma B}x$;
\item[(vi)] $(z_n)_{n\geq 1}$ converges weakly to $J_{\gamma B}x$;
\item[(vii)] if $A$ or $B$ is uniformly monotone, then $(y_n)_{n\geq 1}$ and $(z_n)_{n\geq 1}$ converge strongly to the unique point in $\zer(A+B)$.                    
\end{itemize} 
\end{theorem}

\begin{proof} We use again the notation $w_n=x_n+\alpha_n(x_n-x_{n-1})$ for all $n\geq 1$. Taking into account the iteration rules and the definition of the reflected resolvent, the iterative scheme in the enunciation of 
the theorem can be for every $n \geq 1$ written as
\begin{align}\label{DR-kr-m}x_{n+1}= & \  w_n+\lambda_n\Big[J_{\gamma A}\circ(2J_{\gamma B}-\id)w_n-J_{\gamma B}w_n\Big]\nonumber\\
                    = & \ w_n+\lambda_n\left[\left(\frac{\id +R_{\gamma A}}{2}\circ R_{\gamma B}\right)w_n-\frac{\id +R_{\gamma B}}{2}w_n\right]\nonumber\\
                    = & \ w_n+\frac{\lambda_n}{2}(Tw_n-w_n),
\end{align}
where $T:=R_{\gamma A}\circ R_{\gamma A} : {\cal H} \rightarrow {\cal H}$ is a nonexpansive operator. From \eqref{zer(A+B)} we have $\zer(A+B)=J_{\gamma B}(\fix T)$, hence we get 
$\fix T\neq\varnothing$. By applying Theorem \ref{inertial-kr-m-th}, there exists $x\in\fix T$ such that 
(i)-(iii) hold. 

(iv) Follows from Theorem \ref{inertial-kr-m-th} and relation \eqref{Tw_n}, since $z_n-y_n=\frac{1}{2}(Tw_n-w_n)$ for $n \geq 1$. 

(v) We will show that $(y_n)_{n\geq 1}$ is bounded and that $J_{\gamma B}x$ is the unique weak sequential cluster point of 
$(y_n)_{n\geq 1}$. From here the conclusion will automatically follow. By using that $J_{\gamma B}$ is nonexpansive, for all $n\geq 1$ we have  
$$\|y_n-y_1\|=\|J_{\gamma B}w_n-J_{\gamma B}w_1\|\leq\|w_n-w_1\|=\|x_n-x_1+\alpha_n(x_n-x_{n-1})\|.$$ 
Since $(x_n)_{n\in\N}$ is bounded (by (iii)) and $(\alpha_n)_{n\geq 1}$ is also bounded, the sequence $(y_n)_{n\geq 1}$ is bounded, too. 

Now let $y$ be a sequential weak cluster point of $(y_n)_{n\geq 1}$, that is, the latter has a subsequence $(y_{n_k})_{k \in \N}$
fulfilling $y_{n_k}\rightharpoonup y$ as $k\rightarrow+\infty$. We use the notations $u_n:=2y_n-w_n-z_n$ and $v_n:=w_n-y_n$ 
for all $n\geq 1$. The definitions of the resolvent yields
\begin{equation}\label{res}
(z_n,u_n)\in\gr (\gamma A), \ (y_n,v_n)\in \gr (\gamma B) \ \mbox{and} \ u_n+v_n=y_n-z_n \ \forall n \geq 1.
\end{equation}

Further, by (ii), (iii) and (iv) we derive  
$$ z_{n_k}\rightharpoonup y, w_{n_k}\rightharpoonup x, u_{n_k}\rightharpoonup y-x \mbox{ and } v_{n_k}\rightharpoonup x-y \mbox{ as }k\rightarrow+\infty.$$
Using again (ii) and Lemma \ref{conv-weak-str-sum} we obtain $y\in\zer(\gamma A+\gamma B)=\zer(A+B)$, $(y,y-x)\in\gr \gamma A$ 
and $(y,x-y)\in\gr \gamma B$. As a consequence, $y=J_{\gamma B}x$.

(vi) Follows from (iv) and (v). 

(vii) We prove the statement in case $A$ is uniformly monotone, the situation when $B$ fulfills this condition being similar. 
Denote $y=J_{\gamma B}x$. There exists an increasing function $\phi_A:[0,+\infty)\rightarrow [0,+\infty]$ that 
vanishes only at $0$ such that (see also \eqref{res} and the considerations made in the proof of (v))
$$\gamma\phi_A(\|z_n-y\|)\leq \langle z_n-y,u_n-y+x\rangle \ \forall n\geq 1.$$

Moreover, since $B$ is monotone we have (see \eqref{res})
$$0\leq\langle y_n-y,v_n-x+y\rangle=\langle y_n-y,y_n-z_n-u_n-x+y\rangle \ \forall n\geq 1.$$
Summing up the last two inequalities we obtain
$$\gamma\phi_A(\|z_n-y\|)\leq \langle z_n-y_n,u_n-y_n+x\rangle=\langle z_n-y_n,y_n-z_n-w_n+x\rangle \ \forall n\geq 1.$$

Since $z_n-y_n\rightarrow 0$ and $w_n\rightharpoonup x$ as $n\rightarrow+\infty$, from the last inequality we get
$\lim_{n\rightarrow+\infty}\phi_A(\|z_n-y\|)=0$, hence $z_n\rightarrow y$ and therefore $y_n\rightarrow y$ as 
$n\rightarrow+\infty$. 
\end{proof}

\begin{remark} In case $\alpha=0$, which enforces $\alpha_n=0$ for all $n\geq 1$, the iterative scheme in Theorem \ref{inertial-DR} becomes the classical Douglas--Rachford 
splitting algorithm (see \cite[Theorem 25.6]{bauschke-book}): 
$$(\forall n\geq 1)\hspace{0.2cm}\left\{
\begin{array}{ll}
y_n=J_{\gamma B}x_n\\
z_n=J_{\gamma A}(2y_n-x_n)\\
x_{n+1}=x_n+\lambda_n(z_n-y_n),
\end{array}\right.$$
 
\noindent the convergence of which holds under the assumption $\sum_{n\in\N}\lambda_n(1-\lambda_n)=+\infty$. 
\end{remark}

\begin{remark}\label{inertial-prox} In case $Bx=0$ for all $x\in\cal{H}$, the iterative scheme in Theorem \ref{inertial-DR} becomes
$$x_{n+1} = \lambda_n J_{\gamma A}\big(x_n+\alpha_n(x_n-x_{n-1})\big) + (1-\lambda_n) (x_n+\alpha_n(x_n-x_{n-1})) \ \forall n\geq 1,$$ 
which was already considered in \cite{alvarez2004} as a proximal-point algorithm (see \cite{rock-prox}) in the context of solving the
monotone inclusion problem $0 \in Ax$. Notice that in this scheme in each iteration a constant step-size $\gamma>0$ is considered.  
Proximal-point algorithms of inertial-type with variable step-sizes have been proposed and investigated, for instance,
in \cite[Theorem 2.1]{alvarez-attouch2001}, \cite{alvarez2004} and \cite[Remark 7]{b-c-inertial}.  
\end{remark}

\section{Solving monotone inclusion problems involving mixtures of linearly composed and parallel-sum type operators}\label{sec3}

In this section we apply the inertial Douglas--Rachford algorithm proposed in Section \ref{sec2} to a highly structured primal-dual system of monotone inclusions by making use of appropriate splitting techniques. The problem under investiagtion reads as follows.

\begin{problem}\label{dr_p1}
Let $A:\h \rightrightarrows \h$ be a maximally monotone operator and let $z \in \h$. Further, let $m$ be a strictly positive integer and 
for every $i \in \{1,\!...,m\}$, let $r_i \in \g_i$, $B_i : \g_i \rightrightarrows \g_i$ and $D_i : \g_i \rightrightarrows \g_i$ 
be maximally monotone operators and let $L_i : \h \rightarrow \g_i$ be a nonzero linear continuous operator. The problem is to solve 
the primal inclusion
\begin{align}
	\label{dr_opt-p}
	\text{find }\bx \in \h \text{ such that } z \in A\bx + \sum_{i=1}^m L_i^* (B_i\Box D_i)(L_i \bx-r_i)
\end{align}
together with the dual inclusion
\begin{align}
	\label{dr_opt-d}
	\text{find }\bv_1 \in \g_1,\!...,\bv_m \in \g_m \text{ such that }(\exists x\in\h)\left\{
	\begin{array}{l}
		z - \sum_{i=1}^m L_i^*\bv_i \in Ax \\
		\!\!\bv_i \in \!\! (B_i \Box D_i)(L_ix-r_i), \,i=1,\!...,m.
	\end{array}
\right.
\end{align}
\end{problem}

We say that $(\bx, \bv_1,\!...,\bv_m) \in \h \times \g_1 \,... \times \g_m$ is a primal-dual solution to Problem \ref{dr_p1}, if
\begin{align}
	\label{dr_operator-proof-conditions-full}
	z - \sum_{i=1}^m L_i^*\bv_i \in A\bx \ \mbox{and} \  \bv_i \in (B_i \Box D_i)(L_i\bx-r_i), \ i=1,\!...,m.
\end{align}
Note that, if $(\bx, \bv_1,\!...,\bv_m) \in \h \times \g_1 \,... \times \g_m$ is a primal-dual solution to Problem \ref{dr_p1}, then $\bx$ is a solution to \eqref{dr_opt-p} and $(\bv_1,\!...,\bv_m)$ is a solution to \eqref{dr_opt-d}. On the other hand, if $\bx\in\h$ is a solution to \eqref{dr_opt-p}, then there exists $(\bv_1,\!...,\bv_m)\in\g_1\times \,... \times \g_m$ such that $(\bx, \bv_1,\!...,\bv_m)$ is a primal-dual solution to Problem \ref{dr_p1}. Equivalently, if $(\bv_1,\!...,\bv_m)\in\g_1\times \,... \times \g_m$ is a solution to \eqref{dr_opt-d}, then there exists $\bx\in\h$ such that $(\bx, \bv_1,\!...,\bv_m)$ is a primal-dual solution to Problem \ref{dr_p1}.

Several particular instances of the primal-dual system of monotone inclusions \eqref{dr_opt-p}--\eqref{dr_opt-d} when applied to convex optimization problems can be found in \cite{combettes-pesquet,vu}.

The inertial primal-dual Douglas-Rachford algorithm we would like to propose for solving \eqref{dr_opt-p}--\eqref{dr_opt-d} is formulated as follows.

\begin{algorithm}\label{dr_alg1} Let $x_0,x_1 \in \h$, $v_{i,0},v_{i,1} \in \g_i$, $i=1,\!...,m$, and $\tau, \sigma_i >0$, $i=1,\!...,m,$ be such that
$$\tau \sum_{i=1}^m \sigma_i \|L_i\|^2 < 4. $$
Furthermore, let $(\alpha_n)_{n\geq 1}$ be a nondecreasing sequence with  $\alpha_1=0$ and $0\leq \alpha_n\leq\alpha < 1$ for every $n\geq 1$ and $\lambda, \sigma, \delta >0$ and the sequence $(\lambda_n)_{n\geq 1}$ be such that
\begin{equation*}
\delta>\frac{\alpha^2(1+\alpha)+\alpha\sigma}{1-\alpha^2} \ \mbox{and} \ 0<\lambda\leq\lambda_n\leq 2\frac{\delta-\alpha\Big[\alpha(1+\alpha)+\alpha\delta+\sigma\Big]}{\delta\Big[1+\alpha(1+\alpha)+\alpha\delta+\sigma\Big]} \
\forall n \geq 1.
\end{equation*}
Set
	\begin{align}\label{dr_A1}
	  \left(\forall n\geq 1\right) \begin{array}{l} \left\lfloor \begin{array}{l}
		p_{1,n} = J_{\tau A}\left( x_n +\alpha_n(x_n-x_{n-1}) - \frac{\tau}{2} \sum_{i=1}^m L_i^*(v_{i,n}+\alpha_n(v_{i,n}-v_{i,n-1})) + \tau z \right) \\
		w_{1,n} = 2p_{1,n} - x_n -\alpha_n(x_n-x_{n-1})\\
		\text{For }i=1,\!...,m  \\
				\left\lfloor \begin{array}{l}
					p_{2,i,n} = J_{\sigma_i B_i^{-1}}\left(v_{i,n}+\alpha_n(v_{i,n}-v_{i,n-1}) +\frac{\sigma_i}{2} L_i w_{1,n} - \sigma_i r_i \right) \\
					w_{2,i,n} = 2 p_{2,i,n} - v_{i,n} -\alpha_n(v_{i,n}-v_{i,n-1})
				\end{array} \right.\\
		z_{1,n} = w_{1,n} - \frac{\tau}{2} \sum_{i=1}^m L_i^* w_{2,i,n} \\
		x_{n+1} = x_n +\alpha_n(x_n-x_{n-1}) + \lambda_n ( z_{1,n} - p_{1,n} ) \\
		\text{For }i=1,\!...,m  \\
				\left\lfloor \begin{array}{l}
					z_{2,i,n} = J_{\sigma_i D_i^{-1}}\left(w_{2,i,n} + \frac{\sigma_i}{2}L_i (2 z_{1,n} - w_{1,n}) \right) \\
					v_{i,n+1} = v_{i,n}+  \alpha_n(v_{i,n}-v_{i,n-1}) + \lambda_n (z_{2,i,n} - p_{2,i,n}).
				\end{array} \right. \\ \vspace{-4mm}
		\end{array}
		\right.
		\end{array}
	\end{align}
\end{algorithm}

\begin{theorem}\label{dr-thm} In Problem \ref{dr_p1}, suppose that 
\begin{align}\label{dr_zin}
	z \in \ran \bigg( A +  \sum_{i=1}^m L_i^*(B_i\Box D_i)(L_i \cdot -r_i) \bigg),
\end{align}
and consider the sequences generated by Algorithm \ref{dr_alg1}. Then there exists $(\bx, \bv_1,\!...,\bv_m) \in \h \times \g_1 \, ... \times \g_m$ such that the following statements are true:
\begin{enumerate} 
	\item[(i)]\label{thm1.1} By setting 
	\begin{align*}
	\bp_1 &= J_{\tau A}\left( \bx - \frac{\tau}{2} \sum_{i=1}^m L_i^*\bv_{i} + \tau z \right), \\
	\bp_{2,i} &= J_{\sigma_i B_i^{-1}}\left(\bv_{i}+\frac{\sigma_i}{2} L_i (2\bp_1-\bx) - \sigma_i r_i \right),\ i=1,\!...,m,
	\end{align*}
	the element $(\bp_1,\bp_{2,1},\!...,\bp_{2,m}) \in \h \times \g_1 \times \!... \times \g_m$ is a primal-dual solution to Problem \ref{dr_p1};
	\item[(ii)]\label{thm1.2} $\sum_{n\in\N}\|x_{n+1}-x_n\|^2<+\infty$ and $\sum_{n\in\N}\|v_{i,n+1}-v_{i,n}\|^2<+\infty$, $i=1,\!...,m$; 
	\item[(iii)]\label{thm1.3} $(x_n,v_{1,n},\!...,v_{m,n})_{n\in\N}$ converges weakly to $(\bx, \bv_1,\!...,\bv_m)$;
	\item[(iv)]\label{thm1.4} $(p_{1,n}-z_{1,n},p_{2,1,n}-z_{2,1,n},\!...,p_{2,m,n}-z_{2,m,n})\rightarrow 0$ as $n\rightarrow+\infty$;
	\item[(v)]\label{thm1.5} $(p_{1,n},p_{2,1,n},\!...,p_{2,m,n})_{n\geq 1}$ converges weakly to $(\bp_1,\bp_{2,1},\!...,\bp_{2,m})$;
	\item[(vi)]\label{thm1.6} $(z_{1,n},z_{2,1,n},\!...,z_{2,m,n})_{n\geq 1}$ converges weakly to $(\bp_1,\bp_{2,1},\!...,\bp_{2,m})$;
	\item[(vii)]\label{thm1.7} if $A$ and $B_i^{-1}$, $i=1,\!...,m$, are uniformly monotone, then $(p_{1,n},p_{2,1,n},\!...,p_{2,m,n})_{n\geq 1}$ and $(z_{1,n},z_{2,1,n},\!...,z_{2,m,n})_{n\geq 1}$ converge strongly 
	to the unique primal-dual solution $(\bp_1,\bp_{2,1},\!...,\bp_{2,m})$ to Problem \ref{dr_p1}.                    
\end{enumerate} 
\end{theorem}

\begin{proof}
For the proof we use Theorem \ref{inertial-DR} and adapt the techniques from \cite{b-h2} (see also \cite{vu}) to the given setting. 
We consider the Hilbert space $\fK = \h \times \g_1 \times \,... \times \g_m$ endowed with inner product and associated norm defined, for  $(x,v_1,\!...,v_m)$, $(y,q_1,\!...,q_m) \in \fK$, via
\begin{align}
	\label{dr_def1.001}
	\begin{aligned}
	\< (x,v_1,\!...,v_m),(y,q_1,\!...,q_m) \>_{\fK} &= \<x,y\>_{\h} + \sum_{i=1}^m \< v_i,q_i \>_{\g_i} \\ \text{ and } \|(x,v_1,\!...,v_m)\|_{\fK} &= \sqrt{\|x\|_{\h}^2 + \sum_{i=1}^m \| v_i \|_{\g_i}^2},
	\end{aligned}
\end{align}
respectively. Further, we consider the set-valued operator
\begin{align*}
	\f M : \fK \rightrightarrows \fK, \quad (x,v_1,\!...,v_m) \mapsto (-z + Ax, r_1 + B_1^{-1}v_1, \!..., r_m + B_m^{-1}v_m),
\end{align*}
which is maximally monotone,  since $A$ and $B_i$, $i=1,\!...,m,$ are maximally monotone (see \cite[Proposition 20.22 and Proposition 20.23]{bauschke-book}), and the linear continuous operator
\begin{align*}
	\f S : \fK \rightarrow \fK, \quad (x,v_1,\!...,v_m) \mapsto \left(\sum_{i=1}^m L_i^* v_i, -L_1 x , \!..., - L_m x\right),
\end{align*}
which is skew-symmetric (i.\,e. $\f S^*=-\f S$) and hence maximally monotone (see \cite[Example 20.30]{bauschke-book}). Moreover, we consider the set-valued operator
\begin{align*}
	\f Q : \fK \rightrightarrows \fK, \quad (x,v_1,\!...,v_m) \mapsto \left(0, D_1^{-1} v_1, \!..., D_m^{-1} v_m\right),
\end{align*}
which is once again maximally monotone, since $D_i$ is maximally monotone for $i=1,\!...,m$. 
Therefore, since $\dom \f S = \fK$, both $\frac{1}{2}\f S+ \f Q$ and $\frac{1}{2}\f S + \f M$ are maximally monotone (see \cite[Corollary 24.4(i)]{bauschke-book}). Furthermore, one can easily notice that 
$$\eqref{dr_zin} \Leftrightarrow \,\zer\left(\f M + \f S + \f Q\right) \neq \varnothing$$ 
and
\begin{align}
	\label{dr_optcond}
	\begin{aligned}
	&(x,v_1,\!...,v_m) \in \zer\left(\f M + \f S + \f Q\right) \\ 
	\Rightarrow &(x, v_1,\!...,v_m) \text{ is a primal-dual solution to Problem } \ref{dr_p1}.
	\end{aligned}
\end{align}
We also introduce the linear continuous operator
\begin{align*}
	\f V : \fK \rightarrow \fK, \quad (x,v_1,\!...,v_m) \mapsto \left(\frac{x}{\tau} -\frac{1}{2} \sum_{i=1}^m L_i^* v_i, \frac{v_1}{\sigma_1} - \frac{1}{2} L_1x , \!..., \frac{v_m}{\sigma_m} -\frac{1}{2} L_m x\right),
\end{align*}
which is self-adjoint and $\rho$-strongly positive (see \cite{b-h2}) for 
$$\rho := \left(1-\frac{1}{2} \sqrt{\tau \sum_{i=1}^m \sigma_i \|L_i\|^2}\right) \min\left\{\frac{1}{\tau}, \frac{1}{\sigma_1}, \ldots, \frac{1}{\sigma_m} \right\}>0,$$
namely, the following inequality holds $$\< \f x,\f V\f x\>_{\fK}\geq \rho\|\f x\|_{\fK}^2 \ \forall \f x\in\fK.$$
Therefore, its inverse operator $\f V^{-1}$ exists and it fulfills $\|\f V^{-1}\|\leq \frac{1}{\rho}$.

Note that the algorithmic scheme \eqref{dr_A1} is equivalent to
	\begin{align}\label{dr_A1.1}
	  \left(\forall n\geq 1\right) \!\begin{array}{l} \left\lfloor \begin{array}{l}
		\frac{x_n- p_{1,n}}{\tau} +\alpha_n\frac{x_n - x_{n-1}}{\tau} - \frac{1}{2} \sum_{i=1}^m L_i^*(v_{i,n}+\alpha_n(v_{i,n}-v_{i,n-1})) \in Ap_{1,n}-z \\
		w_{1,n} = 2p_{1,n} - x_n -\alpha_n(x_n-x_{n-1}) \\
		\text{For }i=1,\!...,m  \\
				\left\lfloor \begin{array}{l}
				 \frac{v_{i,n} - p_{2,i,n}}{\sigma_i} +\alpha_n \frac{v_{i,n}-v_{i,n-1}}{\sigma_i}	- \frac{1}{2} L_i (x_n - p_{1,n}+\alpha_n(x_n-x_{n-1})) \\ \hspace{6.5cm} \in -\frac{1}{2}L_i p_{1,n} + B_i^{-1}p_{2,i,n} +r_i  \\
					w_{2,i,n} = 2 p_{2,i,n} - v_{i,n} -\alpha_n(v_{i,n}-v_{i,n-1})\\
				\end{array} \right.\\
		\frac{w_{1,n}-z_{1,n}}{\tau} - \frac{1}{2}\sum_{i=1}^m L_i^* w_{2,i,n} = 0 \\
		x_{n+1} = x_n +\alpha_n(x_n-x_{n-1})+ \lambda_n ( z_{1,n} - p_{1,n} ) \\
		\text{For }i=1,\!...,m  \\
				\left\lfloor \begin{array}{l}
					\frac{w_{2,i,n} - z_{2,i,n}}{\sigma_i} - \frac{1}{2}L_i(w_{1,n}-z_{1,n}) \in -\frac{1}{2}L_i z_{1,n} + D_i^{-1}z_{2,i,n} \\
					v_{i,n+1} = v_{i,n} +\alpha_n(v_{i,n}-v_{i,n-1})+ \lambda_n (z_{2,i,n} - p_{2,i,n}). \\
				\end{array} \right. \\		
		\end{array}
		\right.
		\end{array}
	\end{align}
By introducing the sequences 
\begin{align*}
 \fx_n = (x_n,v_{1,n},\!...,v_{m,n}),\ \fy_n =(p_{1,n},p_{2,1,n},\!...,p_{2,m,n}), \ \fz_n =(z_{1,n},z_{2,1,n},\!...,z_{2,m,n}) \ \forall n \geq 1,
\end{align*}
the scheme \eqref{dr_A1.1} can equivalently be written in the form
\begin{align}
	\label{dr_A1.2}
	\left(\forall n\geq 1\right)  \left\lfloor \begin{array}{l}
	\f V(\fx_n - \fy_n + \alpha_n(\fx_n-\fx_{n-1})) \in \left(\frac{1}{2}\f S +\f M\right)\fy_n \\
	\f V(2 \fy_n - \fx_n- \fz_n - \alpha_n(\fx_n-\fx_{n-1}) ) \in \left(\frac{1}{2}\f S +\f Q\right)\fz_n  \\
	\fx_{n+1} = \fx_n + \alpha_n(\fx_n-\fx_{n-1})+ \lambda_n \left(\fz_n-\fy_n\right),
	\end{array}
	\right.
\end{align}
which is equivalent to 
\begin{align}
	\label{dr_A1.3}
	\left(\forall n\geq 1\right)  \left\lfloor \begin{array}{l}
	\fy_n = \left(\id + \f V^{-1}(\frac{1}{2}\f S +\f M)\right)^{-1}\left(\fx_n + \alpha_n(\fx_n-\fx_{n-1})\right) \\
	\fz_n = \left(\id + \f V^{-1}(\frac{1}{2}\f S +\f Q)\right)^{-1}\left(2 \fy_n - \fx_n - \alpha_n(\fx_n-\fx_{n-1})\right) \\
	\fx_{n+1} = \fx_n + \alpha_n(\fx_n-\fx_{n-1}) + \lambda_n \left(\fz_n-\fy_n\right),
	\end{array}
	\right.
\end{align}
In the following, we consider the Hilbert space $\fK_{\f V}$ with inner product and norm respectively defined, for $\fx,\fy \in \fK$, via
\begin{align}\label{dr_HSKV}
 \< \fx,\fy \>_{\fK_{\f V}} = \< \fx, \f V \fy \>_{\fK} \text{ and } \|\fx\|_{\fK_{\f V}} = \sqrt{\< \fx, \f V \fx \>_{\fK}}.
\end{align}
As the set-valued operators $\frac{1}{2}\f S+ \f M$ and $\frac{1}{2}\f S+ \f Q$ are maximally monotone on $\fK$, the operators
\begin{align}\label{dr_def1.1}
		\f B := \f V^{-1}\left(\frac{1}{2}\f S +\f M\right) \ \mbox{and} \ \f A:=\f V^{-1}\left(\frac{1}{2}\f S +\f Q\right)
\end{align}
are maximally monotone on $\fK_{\f V}$. Moreover, since $\f V$ is self-adjoint and $\rho$-strongly positive, 
weak and strong convergence in $\fK_{\f V}$ are equivalent with weak and strong convergence in $\fK$, respectively.

Taking this into account, it shows that \eqref{dr_A1.3} becomes
\begin{align}
	\label{dr_A1.4}
	\left(\forall n\geq 1\right)  \left\lfloor \begin{array}{l}
	\fy_n = J_{\f B}\left(\fx_n + \alpha_n(\fx_n-\fx_{n-1})\right) \\
	\fz_n = J_{\f A}\left(2 \fy_n - \fx_n - \alpha_n(\fx_n-\fx_{n-1})\right) \\
	\fx_{n+1} = \fx_n + \alpha_n(\fx_n-\fx_{n-1}) + \lambda_n \left(\fz_n-\fy_n\right),
	\end{array}
	\right.
\end{align}
which is the inertial Douglas--Rachford algorithm presented in Theorem \ref{inertial-DR} in the space $\fK_{\f V}$ for $\gamma=1$. Furthermore, we have
$$\zer(\f A+\f B) = \zer(\f V^{-1}\left(\f M+\f S + \f Q\right)) = \zer(\f M + \f S + \f Q).$$

(i) By Theorem \ref{inertial-DR} (i), there  exists $\fbx=(\ol x,\ol v_1,...,\ol v_m)\in\fix(R_{\f A}R_{\f B})$, such that 
$J_{\f B}\fbx\in\zer(\f A+\f B) = \zer(\f M + \f S + \f Q)$. The claim follows by \eqref{dr_optcond} and by identifying $J_{\f B}\fbx$.

(ii) Since $\f V$ is $\rho$-strongly positive, we obtain from Theorem \ref{inertial-DR} (ii) that
\begin{align*}
	\rho \sum_{n\in\N}\|\fx_{n+1}-\fx_n\|_{\fK}^2 \leq \sum_{n\in\N}\|\fx_{n+1}-\fx_n\|_{\fK_{\f V}}^2<+\infty, 
\end{align*}
and therefore the claim follows by considering \eqref{dr_def1.001}.

(iii)--(vi) Follows directly from Theorem \ref{inertial-DR} (iii)--(vi).

(vii) The uniform monotonicity of $A$ and $B_i^{-1}$, $i=1,\!...,m$, implies uniform monotonicity of $\f M$ on $\fK$ (see, for instance, \cite[Theorem 2.1 (ii)]{b-h2}), 
while this further implies uniform monotonicity of $\f B$ on $\fK_{\f V}$. Therefore, the claim 
follows by Theorem \ref{inertial-DR} (vii).
\end{proof}

\section{Convex optimization problems}\label{sec4}

The aim of this section is to show how the inertial Douglas-Rachford primal-dual algorithm can be implemented 
when solving a primal-dual pair of convex optimization problems.

We recall first some notations used in the variational case, see \cite{bo-van, b-hab, bauschke-book, EkTem, simons, Zal-carte}. For a function $f:{\cal H}\rightarrow\overline{\R}$, where $\overline{\R}:=\R\cup\{\pm\infty\}$ is the extended real line, 
we denote by $\dom f=\{x\in {\cal H}:f(x)<+\infty\}$ its \textit{effective domain} and say that $f$ is \textit{proper} if $\dom f\neq\varnothing$ and $f(x)\neq-\infty$ for all $x\in {\cal H}$. 
We denote by $\Gamma({\cal H})$ the family of proper, convex and lower semi-continuous extended real-valued functions defined on ${\cal H}$. 
Let $f^*:{\cal H} \rightarrow \overline \R$, $f^*(u)=\sup_{x\in {\cal H}}\{\langle u,x\rangle-f(x)\}$ for all $u\in {\cal H}$, be the \textit{conjugate function} of $f$. 
The \textit{subdifferential} of $f$ at $x\in {\cal H}$, with $f(x)\in\R$, is the set $\partial f(x):=\{v\in {\cal H}:f(y)\geq f(x)+\langle v,y-x\rangle \ \forall y\in {\cal H}\}$. 
We take by convention $\partial f(x):=\varnothing$, if $f(x)\in\{\pm\infty\}$.  Notice that if $f\in\Gamma({\cal H})$, then $\partial f$ is a maximally monotone operator (see \cite{rock}) and it holds $(\partial f)^{-1} =
\partial f^*$. For two proper functions $f,g:{\cal H}\rightarrow \overline{\R}$, we consider their \textit{infimal convolution}, which is the function $f\Box g:{\cal H}\rightarrow\B$, defined by 
$(f\Box g)(x)=\inf_{y\in {\cal H}}\{f(y)+g(x-y)\}$, for all $x\in {\cal H}$.

Let $S\subseteq {\cal H}$ be a nonempty set. The \textit{indicator function} of $S$, $\delta_S:{\cal H}\rightarrow \overline{\R}$, is the function which takes the value $0$ on $S$ and $+\infty$ otherwise. 
The subdifferential of the indicator function is the \textit{normal cone} of $S$, that is $N_S(x)=\{u\in {\cal H}:\langle u,y-x\rangle\leq 0 \ \forall y\in S\}$, if $x\in S$ and $N_S(x)=\varnothing$ for $x\notin S$.

When $f\in\Gamma({\cal H})$ and $\gamma > 0$, for every $x \in {\cal H}$ we denote by $\prox_{\gamma f}(x)$ the \textit{proximal point} of parameter $\gamma$ of $f$ at $x$, 
which is the unique optimal solution of the optimization problem
\begin{equation}\label{prox-def}\inf_{y\in {\cal H}}\left \{f(y)+\frac{1}{2\gamma}\|y-x\|^2\right\}.
\end{equation}
Notice that $J_{\gamma\partial f}=(\id_{\cal H}+\gamma\partial f)^{-1}=\prox_{\gamma f}$, thus  $\prox_{\gamma f} :{\cal H} \rightarrow {\cal H}$ is a single-valued operator fulfilling the extended
\textit{Moreau's decomposition formula}
\begin{equation}\label{prox-f-star}
\prox\nolimits_{\gamma f}+\gamma\prox\nolimits_{(1/\gamma)f^*}\circ\gamma^{-1}\id\nolimits_{\cal H}=\id\nolimits_{\cal H}.
\end{equation}
Let us also recall that a proper function $f:{\cal H} \rightarrow \overline \R$ is said to be \textit{uniformly convex}, if there exists
an increasing function $\phi :[0,+\infty)\rightarrow[0,+\infty]$ which vanishes only at $0$ and such that
$$f(t x+(1-t)y)+t(1-t)\phi(\|x-y\|)\leq tf(x)+(1-t)f(y) \ \forall x,y\in\dom f\mbox{ and } \forall t\in(0,1).$$
In case this inequality holds for $\phi=(\beta/2)(\cdot)^2$, where $\beta >0$, then $f$ is said to be
\textit{$\beta$-strongly convex}. Let us mention that this property implies $\beta$-strong monotonicity of $\partial f$ (see \cite[Example 22.3]{bauschke-book})
(more general, if $f$ is uniformly convex, then $\partial f$ is uniformly monotone, see \cite[Example 22.3]{bauschke-book}).

Finally, we notice that for $f=\delta_S$, where $S\subseteq {\cal H}$ is a nonempty convex and closed set, it holds
\begin{equation}\label{projection}
J_{\gamma N_S}=J_{N_S}=J_{\partial \delta_S} = (\id\nolimits_{\cal H}+N_S)^{-1}=\prox\nolimits_{\delta_S}=P_S,
\end{equation}
where  $P_S :{\cal H} \rightarrow C$ denotes the \textit{orthogonal projection operator} on $S$ (see \cite[Example 23.3 and Example 23.4]{bauschke-book}).

In the sequel we consider the following primal-dual pair of convex optimization problems.
\begin{problem}\label{dr_p1_convex}
Let $\h$ be a real Hilbert space and let $f \in \Gamma(\h)$, $z\in\h$. Let $m$ be a strictly positive integer and for every 
$i\in\{1,\!...,m\}$, suppose that $\g_i$ is a real Hilbert space, let $g_i,\,l_i\in\Gamma(\g_i)$, $r_i\in\g_i$ and let $L_i:\h \rightarrow \g_i$ be a nonzero bounded linear operator. Consider the convex optimization problem
\begin{align}
	\label{dr_opt-mp}
	(P) \quad \inf_{x \in \h}{\left\{f(x)+\sum_{i=1}^m (g_i \Box l_i)(L_ix-r_i) -\<x,z\>\right\}}
\end{align}
and its conjugate dual problem
\begin{align}
	\label{dr_opt-md}
	(D) \quad \sup_{(v_1,...,v_m) \in \g_1\times\,...\times\g_m}{\left\{-f^*\left( z - \sum_{i=1}^m L_i^*v_i\right) - \sum_{i=1}^m \left( g_i^*(v_i) + l_i^*(v_i) + \< v_i,r_i \> \right) \right\} }.
\end{align}
\end{problem}

By taking into account the maximal monotone operators
\begin{align*}
	 A=\partial f, \ B_i = \partial g_i \text{ and }	D_i = \partial l_i,\ i=1,\!...,m,
\end{align*}
the monotone inclusion problem \eqref{dr_opt-p} reads
\begin{align}
	\label{dr_opt-po}
	\text{find }\bx \in \h \text{ such that } z \in \partial f(\bx) + \sum_{i=1}^m L_i^* (\partial g_i \Box \partial l_i)(L_i \bx-r_i),
\end{align}
while the dual inclusion problem \eqref{dr_opt-d} reads
\begin{align}
	\label{dr_opt-do}
	\text{find }\bv_1 \in \g_1,\,...,\bv_m \in \g_m \text{ such that }(\exists x\in\h)\left\{
	\begin{array}{l}
		z - \sum_{i=1}^m L_i^*\bv_i \in \partial f(x) \\
		\!\!\bv_i \in \!\! (\partial g_i \Box \partial l_i)(L_ix-r_i), \,i=1,\!...,m.
	\end{array}
\right.
\end{align}
If $(\bx, \bv_1,\!...,\bv_m) \in \h \times \g_1 \,... \times \g_m$ is a primal-dual solution to \eqref{dr_opt-po}--\eqref{dr_opt-do}, namely, 
\begin{align}
	\label{dr_operator-proof-conditions-full-0}
	z - \sum_{i=1}^m L_i^*\bv_i \in \partial f(\bx) \ \mbox{and} \  \bv_i \in (\partial g_i \Box \partial l_i)(L_i\bx-r_i), \,i=1,\!...,m,
\end{align}
then $\bx$ is an optimal solution to $(P)$, $(\bv_1,\!...,\bv_m)$ is an optimal solution to $(D)$ and the optimal objective values of the two problems, which we denote by $v(P)$ and $v(D)$, respectively, coincide (thus, strong duality holds).

Combining this statement with Algorithm \ref{dr_alg1} and Theorem \ref{dr-thm} gives rise to the following iterative scheme and corresponding convergence theorem for the primal-dual pair of optimization problems 
$(P)$--$(D)$. 

\begin{algorithm}\label{dr_alg1_convex} Let $x_0,x_1 \in \h$, $v_{i,0},v_{i,1} \in \g_i$, $i=1,\!...,m$, and $\tau, \sigma_i >0$, $i=1,\!...,m,$ be such that
$$\tau \sum_{i=1}^m \sigma_i \|L_i\|^2 < 4. $$
Furthermore, let $(\alpha_n)_{n\geq 1}$ be a nondecreasing sequence with  $\alpha_1=0$ and $0\leq \alpha_n\leq\alpha < 1$ for every $n\geq 1$ and $\lambda, \sigma, \delta >0$ and the sequence $(\lambda_n)_{n\geq 1}$ be such that
\begin{equation*}
\delta>\frac{\alpha^2(1+\alpha)+\alpha\sigma}{1-\alpha^2} \ \mbox{and} \ 0<\lambda\leq\lambda_n\leq 2\frac{\delta-\alpha\Big[\alpha(1+\alpha)+\alpha\delta+\sigma\Big]}{\delta\Big[1+\alpha(1+\alpha)+\alpha\delta+\sigma\Big]} \
\forall n \geq 1.
\end{equation*}
Set
	\begin{align}\label{dr_A1_convex}
	  \left(\forall n\geq 1\right) \begin{array}{l} \left\lfloor \begin{array}{l}
		p_{1,n} = \prox\nolimits_{\tau f}\!\left(x_n +\alpha_n(x_n-x_{n-1}) - \!\frac{\tau}{2} \!\sum_{i=1}^m L_i^*(v_{i,n}+\alpha_n(v_{i,n}-v_{i,n-1})) + \tau z \right) \\
		w_{1,n} = 2p_{1,n} - x_n -\alpha_n(x_n-x_{n-1})\\
		\text{For }i=1,\!...,m  \\
				\left\lfloor \begin{array}{l}
					p_{2,i,n} = \prox\nolimits_{\sigma_i g_i^{*}}\left(v_{i,n}+\alpha_n(v_{i,n}-v_{i,n-1}) +\frac{\sigma_i}{2} L_i w_{1,n} - \sigma_i r_i \right) \\
					w_{2,i,n} = 2 p_{2,i,n} - v_{i,n} -\alpha_n(v_{i,n}-v_{i,n-1})
				\end{array} \right.\\
		z_{1,n} = w_{1,n} - \frac{\tau}{2} \sum_{i=1}^m L_i^* w_{2,i,n} \\
		x_{n+1} = x_n +\alpha_n(x_n-x_{n-1}) + \lambda_n ( z_{1,n} - p_{1,n} ) \\
		\text{For }i=1,\!...,m  \\
				\left\lfloor \begin{array}{l}
					z_{2,i,n} = \prox\nolimits_{\sigma_i l_i^{*}}\left(w_{2,i,n} + \frac{\sigma_i}{2}L_i (2 z_{1,n} - w_{1,n}) \right) \\
					v_{i,n+1} = v_{i,n}+  \alpha_n(v_{i,n}-v_{i,n-1}) + \lambda_n (z_{2,i,n} - p_{2,i,n}).
				\end{array} \right. \\ \vspace{-4mm}
		\end{array}
		\right.
		\end{array}
	\end{align}
\end{algorithm}

\begin{theorem}\label{dr-thm_convex} In Problem \ref{dr_p1_convex}, suppose that 
\begin{align}\label{dr_zin_convex}
	z \in \ran \bigg( \partial f +  \sum_{i=1}^m L_i^*(\partial g_i \Box \partial l_i)(L_i \cdot -r_i) \bigg),
\end{align}
and consider the sequences generated by Algorithm \ref{dr_alg1_convex}. Then there exists $(\bx, \bv_1,\!...,\bv_m) \in \h \times \g_1 \, ... \times \g_m$ such that the following statements are true:
\begin{enumerate} 
	\item[(i)] By setting 
	\begin{align*}
	\bp_1 &= \prox\nolimits_{\tau f}\left( \bx - \frac{\tau}{2} \sum_{i=1}^m L_i^*\bv_{i} + \tau z \right), \\
	\bp_{2,i} &= \prox\nolimits_{\sigma_i g_i^{*}}\left(\bv_{i}+\frac{\sigma_i}{2} L_i (2\bp_1-\bx) - \sigma_i r_i \right),\ i=1,\!...,m,
	\end{align*}
	the element $(\bp_1,\bp_{2,1},\!...,\bp_{2,m}) \in \h \times \g_1 \times \!... \times \g_m$ is a primal-dual solution to 
	Problem \ref{dr_p1_convex}, hence $\bp_1$ is an optimal solution to $(P)$ and $(\bp_{2,1},\!...,\bp_{2,m})$ is an optimal solution to $(D)$;
	\item[(ii)] $\sum_{n\in\N}\|x_{n+1}-x_n\|^2<+\infty$, and $\sum_{n\in\N}\|v_{i,n+1}-v_{i,n}\|^2<+\infty$, $i=1,\!...,m$; 
	\item[(iii)] $(x_n,v_{1,n},\!...,v_{m,n})_{n\in\N}$ converges weakly to $(\bx, \bv_1,\!...,\bv_m)$;
	\item[(iv)] $(p_{1,n}-z_{1,n},p_{2,1,n}-z_{2,1,n},\!...,p_{2,m,n}-z_{2,m,n})\rightarrow 0$ as $n\rightarrow+\infty$;
	\item[(v)] $(p_{1,n},p_{2,1,n},\!...,p_{2,m,n})_{n\geq 1}$ converges weakly to $(\bp_1,\bp_{2,1},\!...,\bp_{2,m})$;
	\item[(vi)] $(z_{1,n},z_{2,1,n},\!...,z_{2,m,n})_{n\geq 1}$ converges weakly to $(\bp_1,\bp_{2,1},\!...,\bp_{2,m})$;
	\item[(vii)] if $f$ and $g_i^*$, $i=1,\!...,m$, are uniformly convex, then $(p_{1,n},p_{2,1,n},\!...,p_{2,m,n})_{n\geq 1}$ and $(z_{1,n},z_{2,1,n},\!...,z_{2,m,n})_{n\geq 1}$ converge strongly to the 
	unique primal-dual solution \newline $(\bp_1,\bp_{2,1},\!...,\bp_{2,m})$ to Problem \ref{dr_p1_convex}.                    
\end{enumerate}
\end{theorem}

We refer the reader to \cite{b-h2,combettes-pesquet} for qualification conditions which guarantee that the inclusion in 
\eqref{dr_zin_convex} holds. Finally, let us mention that for $i=1,...,m$, the function $g_i^*$ is uniformly convex if it is 
$\alpha_i$-strongly convex for $\alpha_i>0$ and this is the case if and only if $g_i$ is Fr\'{e}chet-differentiable with 
$\alpha_i^{-1}$-Lipschitz gradient (see \cite[Theorem 18.15]{bauschke-book}). 

\section{Numerical experiments}\label{sec5}
\subsection{Clustering}\label{sec5.1}
In cluster analysis one aims for grouping a set of points such that points within the same group are more similar to each other than to points in other groups. By taking into account \cite{ChiLan13,LinOhlLju11}, clustering can be formulated as the convex optimization problem
\begin{align}
	\label{sec4_p_cluster}
	\inf_{x_i \in \R^n,\,i=1,\ldots,m} \left\{\frac{1}{2} \sum_{i=1}^m\|x_i-u_i\|^2 + \gamma \sum_{i<j} \omega_{ij}\|x_i-x_j\|_p\right\},
\end{align}
where $\gamma \in \R_{+}$ is a tuning parameter, $p\in\{1,2\}$ and $\omega_{ij}\in\R_+$ represent weights on the terms $\|x_i-x_j\|_p$, for $i,\,j\in\{1,\ldots,m\}$, $i<j$. For each given point $u_i\in\R^n$, $i=1,\ldots,m$, the variable $x_i\in\R^n$ represents the associated cluster center. Since the objective function is strongly convex, there exists a unique solution to \eqref{sec4_p_cluster}.

\begin{figure}[b]
	\floatbox[{\capbeside\thisfloatsetup{capbesideposition={right,top},capbesidewidth=0.55\textwidth}}]{figure}[\FBwidth]
	{\includegraphics*[viewport= 184 303 458 499, width=0.35\textwidth]{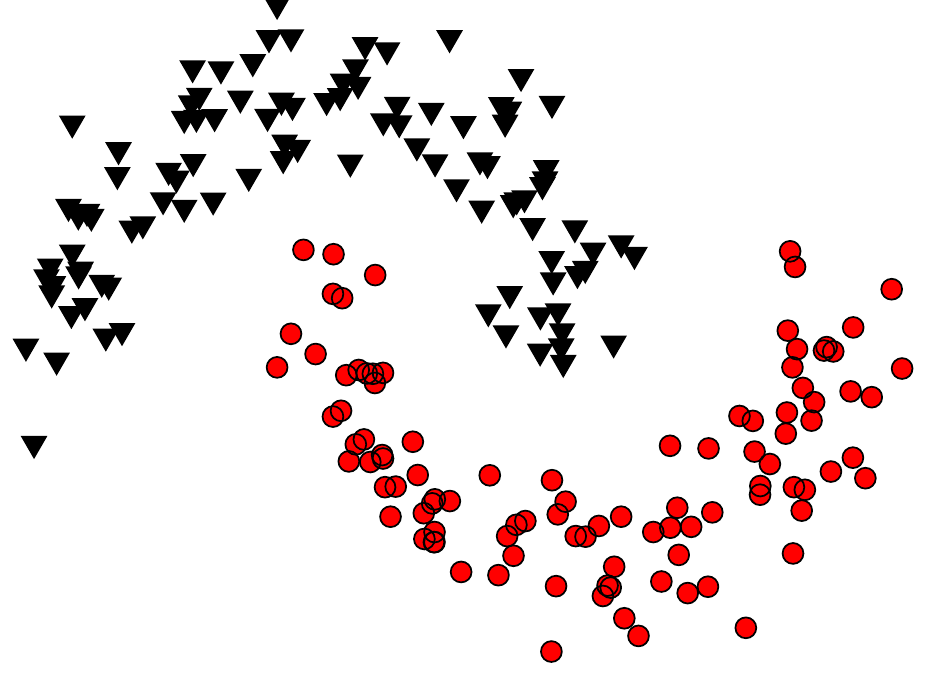}}
	{\caption{Clustering two interlocking half moons. The colors (resp. the shapes) show the correct affiliations.}
	\label{sec4_fig:clusters}}
\end{figure}

The tuning parameter $\gamma \in \R_{+}$ plays a central role within the clustering problem. Taking $\gamma=0$, each cluster center $x_i$ will coincide with the associated point $u_i$. As $\gamma$ increases, the cluster centers will start to coalesce, where two points $u_i$, $u_j$ are said to belong to the same cluster when $x_i=x_j$. One finally obtains a single cluster containing all points when $\gamma$ becomes sufficiently large.

In addition to this, the choice of the weights is important as well, since cluster centers may coalesce immediately as $\gamma$ passes certain critical values. In terms of our weight selection, we use a $K$-nearest neighbors strategy, as proposed in \cite{ChiLan13}. Therefore, whenever $i,\,j\in\{1,\ldots,m\}$, $i<j$, we set the weight to $\omega_{ij}=\iota_{ij}^K \exp(-\phi\|x_i-x_j\|_2^2)$, where
$$ \iota_{ij}^K = \left\{\begin{array}{ll} 1, & \text{if }j \text{ is among }i \text{'s }K\text{-nearest neighbors or vice versa,} \\
																					 0, &\text{otherwise.} \end{array}\right.$$
We use the values $K=10$ and $\phi=0.5$, which are the best ones reported in \cite{ChiLan13} on a similar dataset.

\begin{table}[tb]
	\centering
	\setlength{\tabcolsep}{8pt}
		\begin{tabular}{lllll}
		\toprule
		  & \multicolumn{2}{c}{$p=2$, $\gamma=5.2$} & \multicolumn{2}{c}{$p=1$, $\gamma=4$} \\ 			\cmidrule{2-5}
		     & $\varepsilon=10^{-4}$ & $\varepsilon=10^{-8}$  &  $\varepsilon=10^{-4}$ & $\varepsilon=10^{-8}$ \\ \cmidrule{1-5}
		Algorithm \ref{dr_alg1_convex} & $0.65 \text{s}\ (175)$ 		& $1.36 \text{s}\ (371)$ 	& $0.63 \text{s}\ (176)$ & $1.27 \text{s}\ (374)$ \\
		DR \cite{b-h2} 				& $0.78 \text{s}\ (216)$ 		& $1.68 \text{s}\ (460)$ 	& $0.78 \text{s}\ (218)$ & $1.68 \text{s}\ (464)$ \\
		FB \cite{vu}    			& $2.48 \text{s}\ (1353)$  & $5.72 \text{s}\ (3090)$  & $2.01 \text{s}\ (1092)$& $4.05 \text{s}\ (2226)$ \\
		FB Acc \cite{b-c-h2} 	& $2.04 \text{s}\ (1102)$  & $4.11 \text{s}\ (2205)$  & $1.74 \text{s}\ (950)$ & $3.84 \text{s}\ (2005)$ \\		
		FBF \cite{combettes-pesquet}    &  $7.67 \text{s}\ (2123)$  &  $17.58 \text{s}\ (4879)$  &  $6.33 \text{s}\ (1781)$  &  $13.22 \text{s}\ (3716)$ \\
		FBF Acc \cite{b-h}		& $5.05 \text{s}\ (1384)$  & $10.27 \text{s}\ (2801)$ & $4.83 \text{s}\ (1334)$ & $9.98 \text{s}\ (2765)$  \\
		PD	\cite{ch-pck}			&	$1.48 \text{s}\ (780)$   & $3.26 \text{s}\ (1708)$ 	& $1.44 \text{s}\ (772)$  & $3.18 \text{s}\ (1722)$  \\
		PD Acc	\cite{ch-pck}	& $1.28 \text{s}\ (671)$   & $3.14 \text{s}\ (1649)$ 	& $1.23 \text{s}\ (665)$  & $3.12 \text{s}\ (1641)$  \\
		Nesterov \cite{Nes05b}	& $7.85 \text{s}\ (3811)$  & $42.69 \text{s}\ (21805)$ & $7.46 \text{s}\ (3936)$ & $>190 \text{s}\ (>100000)$ \\
		FISTA \cite{BecTeb09} & $7.55 \text{s}\ (4055)$  & $51.01 \text{s}\ (27356)$ & $6.55 \text{s}\ (3550)$ & $47.81 \text{s}\ (26069)$ \\
		\bottomrule
	\end{tabular}	
	\caption{\small Performance evaluation for the clustering problem. The entries refer to the CPU times in seconds and the number of iterations, respectively, needed in order to attain a root mean squared error for the iterates below the tolerance $\varepsilon$.}
	\label{sec4_table:clustering}
\end{table}

Let $k$ be the number of nonzero weights $\omega_{ij}$. Then, one can introduce a linear operator $A:\R^{mn} \rightarrow \R^{kn}$, such that problem \eqref{sec4_p_cluster} can be equivalently written as
\begin{align}
	\label{sec4_p_reformulated}
	\inf_{x \in \R^{mn}} \left\{h(x) + g(Ax)\right\}, 
\end{align}
the function $h$ being $1$-strongly convex and differentiable with $1$-Lipschitz continuous gradient. Also, by taking $p\in\{1,2\}$, the proximal points with respect to $g^*$ are known to be available via explicit formulae.

For our numerical tests we consider the standard data set consisting of two interlocking half moons in $\R^2$, each of them being composed of $100$ points (see Figure \ref{sec4_fig:clusters}). 
The stopping criterion asks the root-mean-square error (RMSE) to be less than or equal to a given bound $\varepsilon$ which is either $\varepsilon=10^{-4}$ or $\varepsilon=10^{-8}$. 
As tuning parameters we use $\gamma=4$ for $p=1$ and $\gamma=5.2$ for $p=2$ since both choices lead to a correct separation of the input data into the two half moons.

Given Table \ref{sec4_table:clustering}, it shows that Algorithm \ref{dr_alg1_convex} performs better than the noninertial Douglas--Rachford (DR) method proposed in \cite[Algorithm 2.1]{b-h2}. 
One can also see that the inertial Douglas--Rachford algorithm is faster than other popular primal-dual solvers, among them the forward-backward-forward (FBF) method from \cite{combettes-pesquet}, 
and the forward-backward (FB) method from \cite{vu}, where in both methods the function $h$ is processed via a forward step. The accelerated versions of the latter and of the primal-dual (PD) method from \cite{ch-pck} converge in less time than their regular variants, but are still slower than Algorithm \ref{dr_alg1}. Notice that the methods called Nesterov and FISTA are accelerated proximal gradient algorithms which are applied to the Fenchel dual problem to \eqref{sec4_p_reformulated}.

\subsection{The generalized Heron problem}

In the sequel we investigate the \textit{generalized Heron problem} which has been recently investigated in \cite{MorNamSal12a,MorNamSal12b} and where for its solving subgradient-type methods have been proposed.

While the \textit{classical Heron problem} concerns the finding of a point $\bu$ on a given straight line in the plane such that the sum of its distances  to two given points is minimal, the problem that we address here aims to find a point in a closed convex set $\Omega \subseteq \R^n$ which minimizes the sum of the distances to given convex closed sets $\Omega_i \subseteq \R^n$, $i=1,\ldots,m$.

\begin{figure}[tb]
	\centering
	\captionsetup[subfigure]{position=top}
	\subfloat[Problem with optimizer]{\includegraphics*[viewport=138 207 481 590,width=0.32\textwidth]{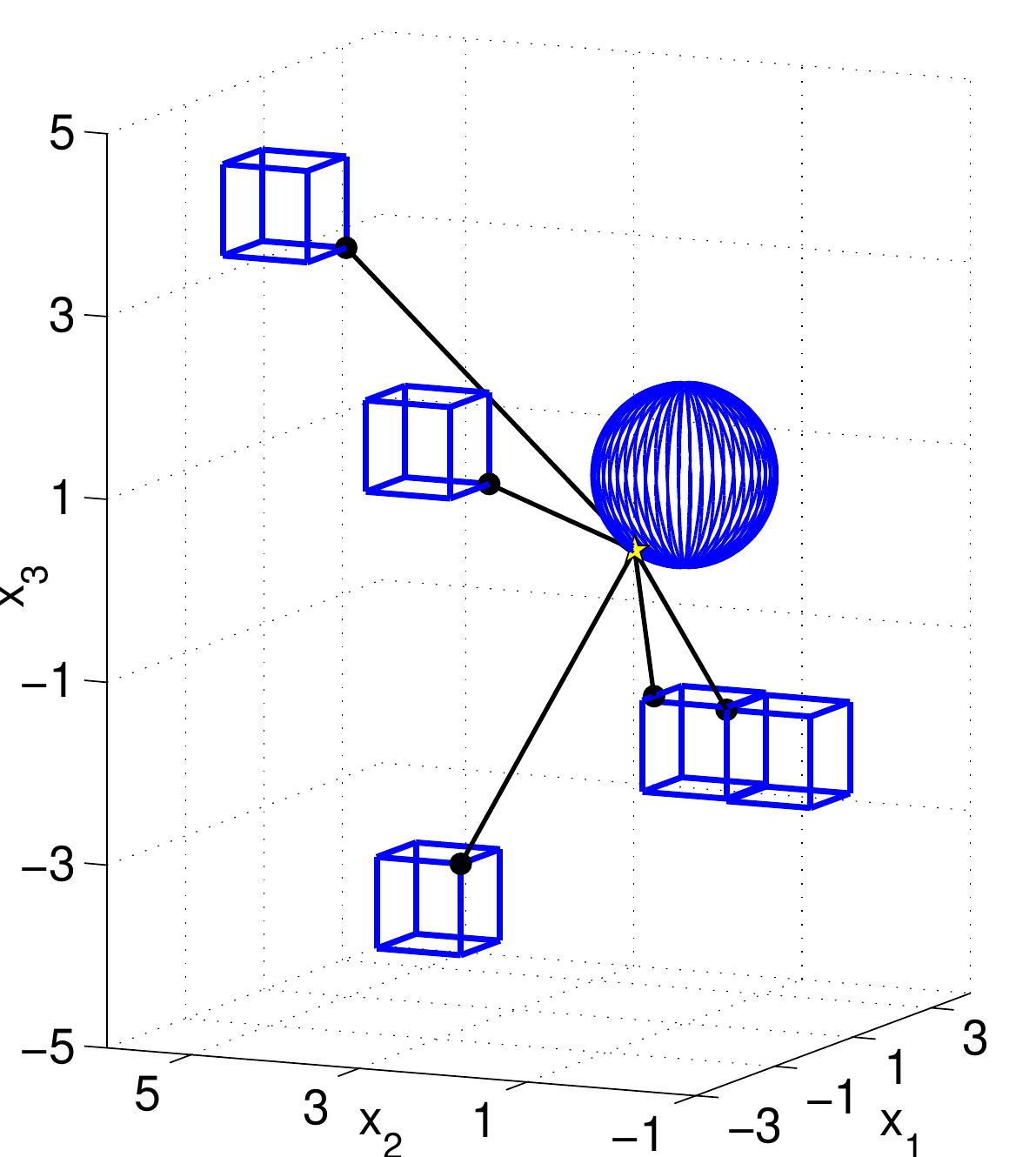}} \hspace{0.8cm}
	\subfloat[Progress of the RMSE values]{\includegraphics*[width=0.48\textwidth]{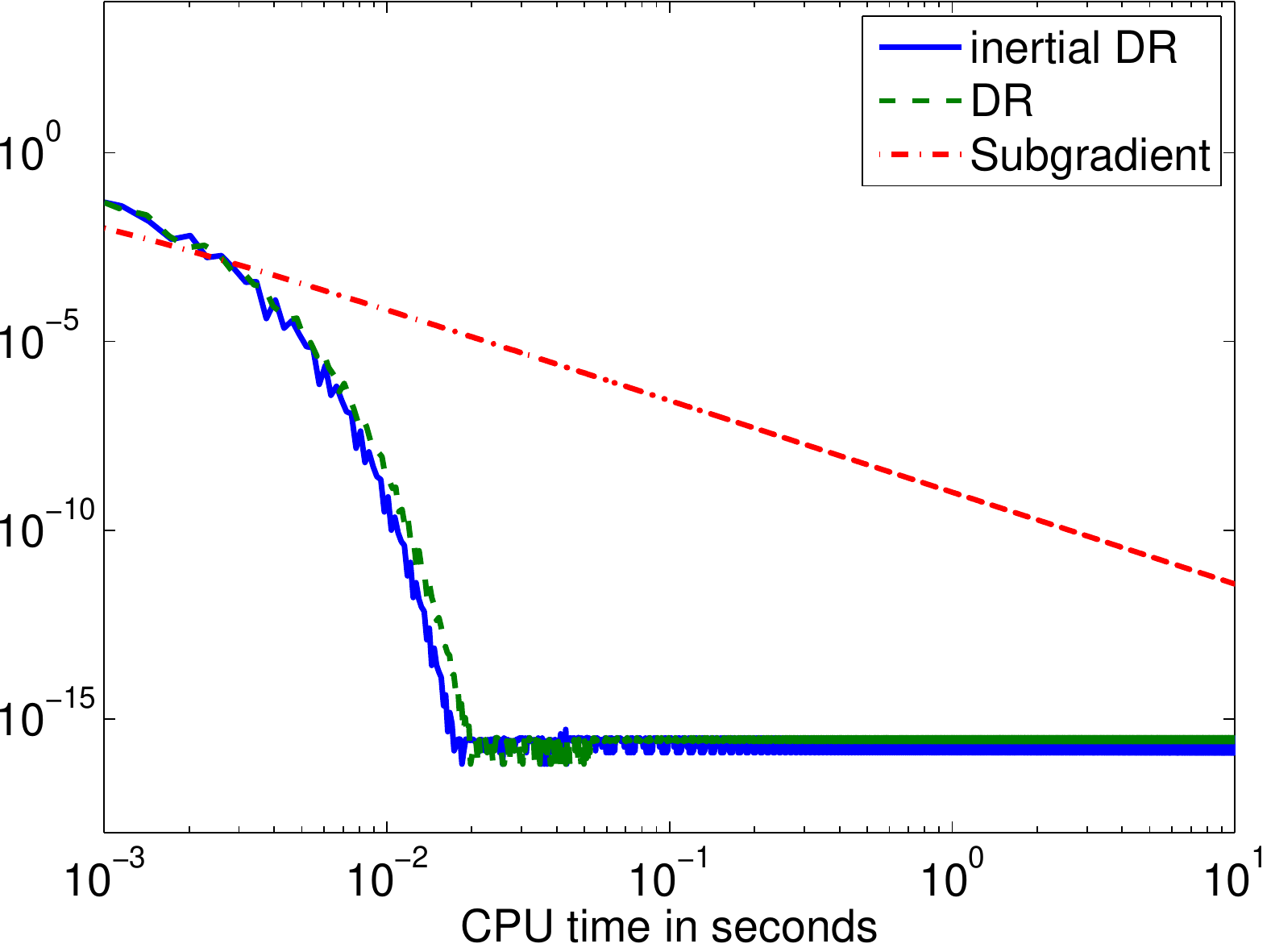}}
	\caption{\small Generalized Heron problem with cubes and ball constraint set on the left-hand side, performance evaluation for the RMSE on the right-hand side.}
	\label{dr_fig:ex1}	
\end{figure}

The distance function from a point $x \in \R^n$ to a nonempty set $\Omega \subseteq \R^n$ is defined as
\begin{align*}
	d(x;\Omega)=(\|\cdot\| \Box \delta_{\Omega})(x) = \inf_{z\in\Omega}\|x-z\|.
\end{align*}
Thus the \textit{generalized Heron problem} reads
\begin{align}
	\label{dr_ex-p1}
	\inf_{x\in\Omega} \sum_{i=1}^m d(x; \Omega_i) ,
\end{align}
where the sets $\Omega \subseteq \R^n$ and $\Omega_i \subseteq \R^n$, $i=1,\ldots,m$, are assumed to be nonempty, closed and convex. We observe that \eqref{dr_ex-p1} perfectly fits into the framework considered in Problem \ref{dr_p1_convex} when setting
\begin{align}
	\label{dr_ex-f1}
	f=\delta_{\Omega}, \text{ and } g_i=\|\cdot\|,\ l_i=\delta_{\Omega_i} \text{ for all }i=1,\ldots,m.
\end{align}
However, note that \eqref{dr_ex-p1} cannot be solved via the primal-dual methods in \cite{combettes-pesquet} and \cite{vu}, which require for each $i=1,\!...,m,$ that either $g_i$ or $l_i$ is strongly convex, unless one 
substantially increases the number of primal and dual variables. Notice that
$$g_i^*:\R^n \rightarrow \overline{\R},\ g_i^*(p)=\sup_{x\in\R^n}\left\{\<p,x\> - \|x\|\right\} = \delta_{\overline{B}(0,1)}(p),\ i=1,\ldots,m,$$
where ${\overline{B}(0,1)}$ denotes the closed unit ball,
 thus the proximal points of $f$, $g_i^*$ and $l_i^*$, $i=1,\ldots,m,$ can be calculated via projections, in case of the latter via Moreau's decomposition formula \eqref{prox-f-star}.

In the following we solve a number of random problems where the closed convex set $\Omega\subseteq \R^n$ will always be the unit ball centered at $(1,\!...,1)^T$. 
The sets $\Omega_i\subseteq\R^n$, $i=1,\!...,m$, are boxes in right position (i.\,e., the edges are parallel to the axes) with side length $1$. The box centers are created via independent identically distributed 
Gaussian entries from $\mathcal{N}(0,n^2)$ where the random seed in Matlab is set to $0$. After determining a solution, the stopping criterion asks the root-mean-square error (RMSE) to be less than or equal to a given bound $\varepsilon$. 

Table \ref{sec4_table:heron} shows a comparison between Algorithm \ref{dr_alg1_convex}, the Douglas--Rachford type method from \cite[Algorithm 3.1]{b-h2}, and the 
subgradient approach described in \cite{MorNamSal12a,MorNamSal12b} when applied to different instances of the generalized Heron problem. One such particular case is displayed in 
Figure \ref{dr_fig:ex1} when $n=3$ and $m=5$, while the evolution of the RMSE values is given there in more detail. Empty cells in Table \ref{sec4_table:heron} indicate that it took more than $60$ 
seconds to pass the stopping criterion. Based on the provided data, one can say that both Algorithm \ref{dr_alg1_convex} and the noninertial Douglas--Rachford type method are performing well in this example and that differences in the computational performance are almost negligible. However, one very interesting observation arises when the dimension of the space is set to $n=3$, as the subgradient approach then becomes better and surpasses both primal-dual methods. 

\begin{table}[tb]
	\footnotesize
	\centering
	\setlength{\tabcolsep}{6pt}
		\begin{tabular}{lllllll}
		\toprule
		  & \multicolumn{2}{c}{Algorithm \ref{dr_alg1_convex}} & \multicolumn{2}{c}{Douglas--Rachford, \cite{b-h2}} & \multicolumn{2}{c}{Subgradient, \cite{MorNamSal12a,MorNamSal12b}}\\ \cmidrule{2-7}
		     & $\varepsilon=10^{-5}$ & $\varepsilon=10^{-10}$  &  $\varepsilon=10^{-5}$ & $\varepsilon=10^{-10}$ & $\varepsilon=10^{-5}$ & $\varepsilon=10^{-10}$ \\ \cmidrule{1-7}
		$n=2,\ m=5$  				& $0.01 \text{s}\ (33)$ 		& $0.03 \text{s}\ (72)$ 	& $0.01 \text{s}\ (30)$ 		& $0.03 \text{s}\ (63)$ & --		& -- \\
		$n=2,\ m=10$  				& $0.01 \text{s}\ (21)$ 		& $0.03 \text{s}\ (59)$ 	& $0.01 \text{s}\ (21)$ 		& $0.02 \text{s}\ (43)$ & $0.01 \text{s}\ (8)$ 		& $0.03 \text{s}\ (120)$ \\		
		$n=2,\ m=20$  				& $0.06 \text{s}\ (295)$ 		& $0.11 \text{s}\ (522)$ 	& $0.11 \text{s}\ (329)$ 		& $0.19 \text{s}\ (583)$ & $0.05 \text{s}\ (204)$ 		& $16.78 \text{s}\ (69016)$ \\
		$n=2,\ m=50$  				& $0.18 \text{s}\ (517)$ 		& $0.45 \text{s}\ (1308)$ 	& $0.22 \text{s}\ (579)$ 		& $0.55 \text{s}\ (1460)$ & $0.04 \text{s}\ (152)$ 		& $4.82 \text{s}\ (19401)$ \\
		\cmidrule{1-7}
		$n=3,\ m=5$  		& $0.01 \text{s}\ (16)$ 		& $0.01 \text{s}\ (37)$ 	& $0.01 \text{s}\ (16)$ 		& $0.01 \text{s}\ (33)$ & $0.02 \text{s}\ (70)$ 		& $2.17 \text{s}\ (8081)$ \\
		$n=3,\ m=10$  				& $0.01 \text{s}\ (37)$ 		& $0.03 \text{s}\ (91)$ 	& $0.01 \text{s}\ (41)$ 		& $0.03 \text{s}\ (101)$ & $0.01 \text{s}\ (11)$ 		& $0.03 \text{s}\ (199)$ \\		
		$n=3,\ m=20$  				& $0.01 \text{s}\ (22)$ 		& $0.03 \text{s}\ (52)$ 	& $0.01 \text{s}\ (25)$ 		& $0.03 \text{s}\ (59)$ & $0.01 \text{s}\ (6)$ 		& $0.01 \text{s}\ (32)$ \\
		$n=3,\ m=50$  				& $0.01 \text{s}\ (19)$ 		& $0.02 \text{s}\ (44)$ 	& $0.01 \text{s}\ (21)$ 		& $0.02 \text{s}\ (51)$ & $0.01 \text{s}\ (10)$ 		& $0.01 \text{s}\ (17)$ \\
		\bottomrule
	\end{tabular}	
	\caption{\small Performance evaluation for the Heron problem. The entries refer to the CPU times in seconds and the number of iterations, respectively, needed in order to attain a root-mean-square error lower than the tolerance $\varepsilon$.}
	\label{sec4_table:heron}
\end{table}

\end{document}